\documentclass[11pt]{amsart}
\usepackage{amsmath,amsthm,amscd,amsfonts, amssymb,mathrsfs}
\usepackage{graphicx}
\usepackage{tikz}
\usepackage{color}
\include{diagxy}
\usepgflibrary{shapes}
\newcommand{\sect}[1]{\section{#1}\setcounter{equation}{0}}

\font\mbn=msbm10 scaled \magstep1
\font\mbs=msbm7 scaled \magstep1
\font\mbss=msbm5 scaled \magstep1
\newfam\mbff
\textfont\mbff=\mbn
\scriptfont\mbff=\mbs
\scriptscriptfont\mbff=\mbss

\newcommand{\Di}      {\mathbb{D}}

\newcommand{\N}       { \mathbb{N}}
\newcommand{\Z}        {\mathbb{Z}  } 
\newcommand\Co           {{\mathbb C}}

\newtheorem{Th}{Theorem}[section]
\newtheorem{Lm}[Th]{Lemma}

\newtheorem{Prop}[Th]{Proposition}
\newtheorem{R}[Th]{Remark}
\newtheorem{E}[Th]{Example}
\begin{document}
\title[On Exponential Factorizations of Matrices over Banach Algebras]{On Exponential Factorizations of Matrices over Banach Algebras}

\author{Alexander Brudnyi} 
\address{Department of Mathematics and Statistics\newline
\hspace*{1em} University of Calgary\newline
\hspace*{1em} Calgary, Alberta\newline
\hspace*{1em} T2N 1N4}
\email{albru@math.ucalgary.ca}

\thanks{Research supported in part by NSERC}

\begin{abstract}
We study exponential factorization of invertible matrices over unital  complex Banach algebras. In particular, we prove that every invertible matrix with entries in the algebra of holomorphic functions on a closed bordered Riemann surface can be written as a product of two exponents of matrices over this algebra. Our result extends similar results proved earlier in \cite{KS} and \cite{L} for $2\times 2$ matrices.
\end{abstract}

%\date{}

\maketitle

\sect{Formulation of the Main Results}
{\bf 1.1.} The paper deals with the problem of optimal exponential factorization of invertible matrices over unital  complex Banach algebras. Earlier similar problems were studied in \cite{MR}, \cite{DK}, \cite{KS}, \cite{L}.
 To formulate the results, first we introduce the required definitions.

 For an associative algebra $\mathscr A$ with $1$ over a field of characteristic zero by  $M_n(\mathscr A)$ we denote the $n\times n$ matrix ring over $\mathscr A$, by $GL_n(\mathscr A)\subset M_{n}(\mathscr A)$ the group of invertible matrices, and if, in addition, $\mathscr A$ is commutative, by $SL_n(\mathscr A)\subset GL_n(\mathscr A)$ we denote the subgroup of matrices with determinant $1$. By $E_n(\mathscr A)$ we denote the subgroup of $GL_n(\mathscr A)$ generated by all {\em elementary matrices}, i.e., matrices in $GL_n(\mathscr A)$ which differ from the identity matrix $I_n$ by at most one non-diagonal entry. Let $t_n(\mathscr A)$ be the minimal $t$ such that every matrix in $E_n(\mathscr A)$ is a product of $t$ {\em unitriangular matrices} in $GL_n(\mathscr A)$ (i.e., either upper triangular matrices with $1$ along the main diagonal or lower triangular matrices with $1$ along the diagonal). We write $t_n(\mathscr A)=\infty$ if such $t$ does not exist. 

We associate with every $A\in M_n(\mathscr A)$ the  exponential series
\begin{equation}\label{e1}
\exp\, A:=\sum_{i=0}^\infty \frac{1}{i!}A^i,\qquad A^0:=I_n.
\end{equation}
If  $A$ is nilpotent (i.e., $A^k=0$ for some $k\in\N$), then the series contains finitely many terms and $\exp\, A\in GL_n(\mathscr A)$. Similarly \eqref{e1} determines an element  $\exp\, A\in GL_n(\mathscr A)$ for every $A\in M_n(\mathscr A)$ if $\mathscr A$ is a unital $m$-convex Fr\'echet algebra (see, e.g., \cite{M} for the definition) and, in particular, if $\mathscr A$ is a unital Banach algebra. Here we assume that the algebra $M_n(\mathscr A)$ is equipped with the corresponding operator norm.

Let ${\rm Exp}_n(\mathscr A)\subset GL_n(\mathscr A)$ denote the subgroup generated by exponents of matrices from $M_n(\mathscr A)$. Since every unipotent $A\in GL_n(\mathscr A)$ (i.e., such that $A-I_n\in M_n(\mathscr A)$ is nilpotent) is the exponent of 
\begin{equation}\label{e2}
\log\, A:=\sum_{i=1}^\infty \frac{(-1)^i}{i}(A-I_n)^i,
\end{equation} 
every matrix in $E_n(\mathscr A)$ can be written as a product of finitely many exponents, i.e., $E_n(\mathscr A)\subset {\rm Exp}_n(\mathscr A)$. We are interested in exponential factorizations of matrices in ${\rm Exp}_n(\mathscr A)$ with the least number of factors. To this end,
let $e_n(\mathscr A)$ be the minimal number $t$ such that every matrix in ${\rm Exp}_n(\mathscr A)$ is a product of $t$ exponents and write $e_n(\mathscr A)=\infty$ if such $t$ does not exist. 

From now on we assume that $\mathscr A$ is a commutative unital complex Banach algebra, unless stated otherwise. In this case, $GL_n(\mathscr A)$ is a complex Banach Lie group and  ${\rm Exp}_n(\mathscr A)$ coincides with its connected component  containing $I_n$.

Let $\mathfrak M(\mathscr A)$ be the maximal ideal space of $\mathscr A$, i.e., the set of all nonzero homomorphisms $\mathscr A\rightarrow\mathbb C$.
Since norm of each $\varphi\in \mathfrak M(\mathscr A)$ is at most one, 
$\mathfrak M(\mathscr A)$ is a subset of the closed unit ball of the dual space $\mathscr A^{*}$. It is a compact Hausdorff space in the weak$^*$ topology induced by
$\mathscr A^{*}$ (the {\em Gelfand topology}). Let $C(\mathfrak M(\mathscr A))$ be
the Banach algebra of continuous complex-valued functions on $\mathfrak M(\mathscr A)$ equipped with supremum norm. An element $a\in \mathscr A$ can be thought of as a
function in $C(\mathfrak M(\mathscr A))$ via the {\em Gelfand transform}
$\, \hat{} : \mathscr A\rightarrow
C(\mathfrak M(\mathscr A))$, $\hat a(\varphi):=\varphi(a)$. The map
$\,\hat{}\ $ is a nonincreasing-norm morphism of Banach algebras. It is injective if $\mathscr A$ is {\em semisimple} (i.e., the {\em Jacobson radical} of $\mathscr A$ is trivial).

Recall that for a normal topological space $X$, $\text{dim}\, X\le d$ if every  open cover of $X$ can be refined by an open cover whose order $\le d + 1$. If $\text{dim}\, X\le d$ and
the statement $\text{dim}\, X\le d-1$ is false, we say that $\text{dim}\, X = d$.

It was proved in \cite[Th.\,4]{V2} that for $\mathscr A=C(X)$, the algebra of complex-valued continuous functions on a $d$-dimensional normal topological space $X$,  group $E_n(\mathscr A)$ coincides with the set of {\em null-homotopic} maps, i.e., maps in $SL_n(\mathscr A):=C(X, SL_n(\mathbb C))$ homotopic (in this class of maps) to a constant map, and, moreover, there exists a constant $v(d)\in\mathbb N$ depending on $d$ only such that 
\begin{equation}\label{eq3}
\sup_n t_n(\mathscr A)\le v(d).
\end{equation}
It is known that $v(0)=v(1)=4$ and $v(2)=5$ (cf. \cite[Sect.\,5.2]{IK}). The problem about numerical upper bounds for $v(d)$ with $d\ge 3$ is still opened.

Next recall that an associative unital ring $\mathcal R$ has the {\em Bass stable rank} $1$ if for every pair $(a,b)$ of elements in $\mathcal R$ such that $\mathcal Ra+\mathcal Rb=\mathscr A$ there exists $t\in \mathcal R$ such that $\mathcal R(a+tb)=\mathcal R$ (i.e., $a+tb$ is left invertible). For some examples and properties of rings of the Bass stable rank $1$ see \cite{V3}. 
We are ready to formulate our first result.
\begin{Th}\label{te1}
Let $\mathscr A$ be a unital commutative complex Banach algebra and $d:={\rm dim}\, \mathfrak M(\mathscr A)$. 
\begin{itemize}
\item[(1)]
If $d<\infty$, then 
\[
\sup_n\, e_n(\mathscr A)\le \mbox{$\left\lfloor\frac{v(d)+1}{2}\right\rfloor$}+3\quad {\rm and}\quad \sup_{n\ge n(d)}e_n(\mathscr A)\le 4,
\]
where $n(d):=\left(\left\lfloor\frac d2\right\rfloor +1\right)\cdot \left(\left\lfloor\frac{v(d)}{2}\right\rfloor+3\right)$. \smallskip
\item[(2)]  
\[
\begin{array}{c}
e_n(\mathscr A)=1\ {\rm for\ all\ } n\ {\rm if}\  d=0,\ \, e_n(\mathscr A)\le 3\  {\rm for\ all\ } n\ {\rm if}\ d=1,\  {\rm and}\\
\\
\displaystyle
 \left\{\begin{array}{l}\displaystyle\sup_{n\le 9}\, e_n(\mathscr A)\le 5\smallskip\\
\displaystyle \sup_{n\ge 10}e_n(\mathscr A)\le 4\end{array}\right.
\ \, {\rm if}\ \, d=2.
\end{array}
\]
\item[(3)]
If the  Bass stable rank of $\mathscr A$ is $1$, then $e_n(\mathscr A)\le 3$  for all $n$.\smallskip
\item[(4)]  If $d\ne 0$, then
\[
e_n(\mathscr A)\ge 2\   {\rm for\ all\ } n\ge 2.
\]
\end{itemize}
\end{Th}
\begin{R}
{\rm  If $d=1$, then the  Bass stable rank of $\mathscr A$ is $1$, see \cite[Thm.\,2.9]{CS}, and the stated inequality follows from part (3). In turn, part (3)  is valid for all unital associative algebras, see \cite[Lm.\,9]{DV} and Lemma \ref{lem2.1} below.
}
\end{R}
{\bf 1.2.} In this part, we further investigate the case of  unital commutative complex Banach algebras $\mathscr A$ having the  Bass stable rank $1$. As follows from parts (3) and (4) of Theorem \ref{te1}, in this case 
\begin{equation}\label{eq1.4}
2\le e_n(\mathscr A)\le 3\  {\rm  for\ all\ } n\ge 2 \mbox{ provided that  ${\rm dim}\, \mathfrak M(\mathscr A)\ne 0$}.
\end{equation}

Further, it was shown in \cite[Thm.2]{KS} that $e_2(\mathscr A)= 2$ if $\mathscr A$ is the disk algebra, and it was recently proved in \cite[Thm.\,1]{L} (answering a question posed in \cite{KS}) that the same  is valid for algebras of holomorphic functions on closed bordered Riemann surfaces. It is known that these algebras have the Bass stable rank $1$. In our next theorem, we extend this result to all $n\ge 2$.
We require the following definition. 

For $a\in\mathscr A$ we define its {\em zero locus} by the formula
\[
Z_a:=\{\varphi\in\mathfrak M(\mathscr A)\, :\, \varphi(a)=0\}\subset\mathfrak M(\mathscr A).
\]
\begin{Th}\label{te2}
Let a unital commutative complex Banach algebra $\mathscr A$ satisfy the property: 
\begin{itemize}
\item[{\rm (*)}] For every nonzero $a\in \mathscr A$ the first \v{C}ech cohomology group $H^1(Z_a,\mathbb Z)=0$.
\end{itemize}
Then
$e_n(\mathscr A)=2$ for all $n\ge 2$.
\end{Th}
Note that every algebra $\mathscr A$ satisfying property (*) has the Bass stable rank $1$, see, e.g., 
\cite[Prop.2.1]{CS}. Theorem \ref{te2} raises a question about optimality of bound $3$ in \eqref{eq1.4}.
\begin{E}\label{ex1.3}
{\rm (1) Let $R$ be a compact bordered Riemann surface with the boundary $\partial R$ and $A(R)$ be the Banach algebra of continuous functions of $R$ holomorphic on $R\setminus\partial R$ equipped with supremum norm.
It is well known that $\mathfrak M(A(R))=R$. Further, for every nonidentical zero function $f\in A(R)$, the zero locus  $Z_f$ is either void or consists of a (possibly empty) discrete subset of $R\setminus\partial R$ and a totally disconnected subset of $\partial R$. Hence, $Z_f$ is a totally disconnected subset of $R$ and so  $H^1(Z_f,\mathbb Z)=0$. Thus Theorem \ref{te2} implies that $e_n(A(R))=2$, $n\ge 2$ extending the corresponding result of  \cite{L}.\smallskip

\noindent (2) Theorem \ref{te2} is valid for algebras $\mathscr A$ of the Bass stable rank $1$ such that $H^{1}(\mathfrak M(\mathscr A),\mathbb Z)=0$. This class  includes, e.g., the algebra $C(X)$, where $X$ is a projective limit of a family of compact spaces homeomorpic to finite trees or the Wiener algebra $W(\mathbb R)_{\Gamma}$ of continuous almost periodic functions  on $\mathbb R$ having the Bohr--Fourier spectra in a semigroup $\Gamma\subset\mathbb Q_+$ with sequences of the Bohr--Fourier coefficients in $\ell^1(\Gamma)$, see, e.g., \cite{BRS} for definitions. (It easily follows from \cite[Thm.\,1]{T} that the Bass stable rank of $W(\mathbb R)_{\Gamma}$ is $1$ and from \cite[Thm.\,1.2]{BRS} that $H^{1}(\mathfrak M(W(\mathbb R)_{\Gamma}),\mathbb Z)=0$.)\smallskip

\noindent (3) Theorem \ref{te2} is valid for algebras $\mathscr A$ such $\mathfrak M(\mathscr A)$ is a projective limit of a family of compact spaces homeomorpic to the unit circle as for each nonzero $a\in\mathscr A$ the zero locus $Z_a$ is the projective limit  (under the projective limit construction) of a family of compact subsets homeomorphic to proper subsets of the unit circle and hence having trivial first  \v{C}ech cohomology groups. This implies that $H^1(Z_a,\mathbb Z)=0$ as well.
For instance, the Wiener algebra $W(\mathbb R)_{\Gamma}$ of continuous almost periodic functions  on $\mathbb R$ having the Bohr--Fourier spectra in a subgroup $\Gamma\subset\mathbb Q$ has this property.
}
\end{E}
{\bf 1.3.} The proof of Theorem \ref{te2} is based on a general result of independent interest presented in this part. 

By $S_n\subset\mathbb C$ we denote the set of roots of order $n$ of $1$. A closed $\epsilon$-neighbourhood of $S_n$ is defined as
\[
N_\varepsilon=\bigl\{z\in\mathbb C\, :\, \min_{y\in S_n}\, |z-y|\le\varepsilon\bigr\}.
\]
Let $\mathscr A$ be a (not necessarily commutative) unital complex Banach algebra. We regard $M_n(\mathscr A)$ as the unital complex Banach algebra of $\mathscr A$-linear operators on $\mathscr A^n$ equipped with the operator norm; here for $a=(a_1,\dots , a_n)\in \mathscr A^n$ its norm $\|a\|_{\mathscr A^n}:=\max_{i}\|a_i\|_{\mathscr A}$. By $\sigma_n(A)\Subset\mathbb C$ we denote the spectrum of $A\in M_n(\mathscr A)$.
\begin{Th}\label{te2.5}
Suppose $A=(a_{ij})\in GL_n(\mathscr A)$, $n\ge 2$, is a (upper or lower) triangular matrix  with the product of diagonal elements $a_{11}\cdots a_{nn}=1$. Then for every $\varepsilon>0$ there are matrices $B_1, B_2\in M_n(\mathscr A)$ such that 
\[
A= \exp\, B_1\cdot \exp\, B_2\quad{\rm and}\quad \sigma_n(\exp\, B_1)=S_n,\quad \sigma_n(\exp B_2)\subset N_\varepsilon.
\]
\end{Th}

\sect{Proof of Theorem \ref{te1}}
\begin{proof}
Let $\Sigma_n(\mathscr A)\subset GL_n(\mathscr A)$ be the set of matrices $A$ such that $0$ belongs to the unbounded connected component of $\mathbb C\setminus\sigma_n(A)$. Then $\Sigma_n(\mathscr A)$ is an open subset of $GL_n(\mathscr A)$  contained in the range of  $\exp: M_n(\mathscr A)\rightarrow GL_n(\mathscr A)$, see, e.g., \cite[Thms.\,10.20,\,10.30]{R}. By
$e_n^0(A)$, $A\in {\rm Exp}_n(\mathscr A)$, we denote the minimal number $t$ such that $A$ is a product of $t$ matrices from $\Sigma_n(\mathscr A)$ and a diagonal matrix from ${\rm Exp}_1(\mathscr A)\cdot I_n$. We set 
\begin{equation}\label{e2.1}
e_n^0(\mathscr A)=\sup_{A\in {\rm Exp}_n(\mathscr A)}e_n^0(A).
\end{equation}
Clearly, the function $e_n^0: {\rm Exp}_n(\mathscr A)\rightarrow\N$ is upper semicontinuous and 
\begin{equation}
e_n(\mathscr A)\leq e_n^0(\mathscr A).
\end{equation}
Thus it suffices to prove Theorem \ref{te1} for $e_n^0(\mathscr A)$.

Recall that $t_n(\mathscr A)$ denotes the minimal $t\in\N\cup\{\infty\}$ such that every matrix in $E_n(\mathscr A)$ is a product of $t$ unitriangular matrices in $SL_n(\mathscr A)$. In our case, $E_n(\mathscr A)={\rm Exp}_n(\mathscr A)\cap SL_n(\mathscr A)$ and each matrix in ${\rm Exp}_n(\mathscr A)$ can be written as a product of matrices from $E_n(\mathscr A)$ and  ${\rm Exp}_1(\mathscr A)\cdot I_n$.
\begin{Lm}\label{lem2.1}
The following inequality holds:
\begin{equation}\label{e3}
e_n^0(\mathscr A)\le \mbox{$\left\lfloor\frac{t_n(\mathscr A)}{2}\right\rfloor$} +1.
\end{equation}
\end{Lm}
\begin{proof}
Suppose $A\in E_n(\mathscr A)$, $A=A_1\cdots A_t$, $t\ge 2$, where all $A_j$ are alternating upper and lower unitriangular matrices in $E_n(\mathscr A)$. Then
\[
A=\left(\prod_{i=1}^{\left\lfloor\frac{t}{2}\right\rfloor}C_iA_{2i}C_i^{-1}\right)\cdot C_{\lfloor\frac{t+1}{2}\rfloor},\quad {\rm where}\quad C_i:=A_1\cdots A_{2i-1},\quad 1\le i\le \mbox{$\left\lfloor\frac{t}{2}\right\rfloor$}.
\]
This  and the definition of $e_n^0(\mathscr A)$ give \eqref{e3} as all terms in the product are unipotent (i.e., belong to $\Sigma_n(\mathscr A)$).
\end{proof}

Let $\mathscr J_{\mathscr A}\subset \mathscr A$ be the Jacobson radical and $q: \mathscr A\rightarrow \mathscr A/\mathscr J_\mathscr A=:\mathscr A_s$ be the quotient homomorphism. 
Then $\mathscr A_s$ is a semisimple Banach algebra with respect to the quotient norm. We extend $q$ to a homomorphism $q_n:M_n(\mathscr A)\rightarrow M_n(\mathscr A_s)$ defined by the application of $q$ to  entries of matrices in $M_n(\mathscr A)$.  Then  $\sigma_n(A)=\sigma_n(q_n(A))$ for all $A\in M_n(\mathscr A)$. This implies that $q_n$ maps $\Sigma_n(\mathscr A)$ surjectively onto $\Sigma_n(\mathscr A_s)$ and  $q_n^{-1}(\Sigma_n(\mathscr A_s))=\Sigma_n(\mathscr A)$. Hence, $e_n^0(A)=e_n^0(q_n(A))$ for all $A\in {\rm Exp}_n(\mathscr A)$ and
\begin{equation}\label{e2.4}
e_n^0(\mathscr A)=e_n^0(\mathscr A_s).
\end{equation}

Since maximal ideal spaces of $\mathscr A$ and $\mathscr A_s$ are homeomorphic, inequalities of parts (1) and (2) for $d>0$ of Theorem \ref{te1} follow directly from similar inequalities of \cite[Thm.\,1.1]{Br} for $t_n(\mathscr A_s)$ by means of Lemma \ref{lem2.1} and \cite[Lms.\,7,\,9,\,Thm.\,20]{DV}, and \cite[Thm.\,2.9]{CS} (see the proofs in \cite{Br} for details).\smallskip

Now let us prove part (2) of the theorem for $d=0$. In this case, by the definition of covering dimension, $\mathfrak M(\mathscr A)$ is a totally disconnected space. Hence, the base of topology of $\mathfrak M(\mathscr A)$ consists of clopen subsets. The required result will follow from
\begin{Lm}\label{lem2.2}
Given $A\in GL_n(\mathscr A)$ there exist idempotents $e_1,\dots, e_k\in\mathscr A$ satisfying $\sum_{i=1}^k e_i=1$ and $e_ie_j=0$ for all $i\ne j$ such that every matrix $e_i(A-I_n)+I_n\in\Sigma_n(\mathscr A)$.
\end{Lm}
\begin{proof}
The spectrum $\sigma(A)$  of $A$ is given by the formula
\[
\sigma(A)=\{\lambda\in\mathbb C\, :\, \exists\, y\in \mathfrak M(\mathscr A)\ {\rm such\ that}\ p(\lambda, y):={\rm det}(\hat A(y)-\lambda\cdot \hat{I}_n)=0\}.
\]
Here $\hat I_n\in GL_n(\mathbb C)$ is the identity matrix and $p$ is a continuous function on $\mathbb C\times \mathfrak M(\mathscr A)$ which is a holomorphic polynomial of degree $n$ in $\lambda$. 

Given $x\in  \mathfrak M(\mathscr A)$ we write $p=q_x+r_x$, where $q_x(\lambda):={\rm det}(\hat A(x)-\lambda\cdot \hat I_n)$ and $r_x$ is a continuous function on $\mathbb C\times \mathfrak M(\mathscr A)$ equals zero on $\mathbb C\times\{x\}$. The polynomial $q_x$ has $n$ complex zeros counted with their multiplicities.  Let $s>0$ be so small that the closed $s$-neighbourhood $N_s$ of the set of zeros of $q_x$ consists of mutually disjoint closed disks of radius $s$ and does not contain $0$. Let $\partial N_s$ be the boundary of $N_s$. Since $r_x$ is continuous on the compact set $\partial N_s\times \mathfrak M(\mathscr A)$ and equals zero on $\partial N_s\times\{x\}$, there is a clopen set $U_x\subset  \mathfrak M(\mathscr A)$ containing $x$ such that
 \[
 \sup_{(\lambda,y)\in\partial N_s\times U_x} |r_x(\lambda,y)|<\inf_{\lambda\in \partial N_s}|q_x(\lambda)|.
 \]
Therefore by the Rouch\'e theorem,  for every $y\in U_x$ all zeros of $p(\cdot, y)$ are situated in $N_s$. 

Next, by the Shilov idempotent theorem \cite{S} there is an (unique) idempotent $e_{U_x}\in\mathscr A$ such that the Gelfand transform $\hat{e}_{U_x}\in C(\mathfrak M(\mathscr A))$ of $e_{U_x}$ is the characteristic function of $U_x$. We put
\begin{equation}\label{equ2.5}
A_{U_x}:=e_{U_x}(A-I_n)+I_n.
\end{equation}
By the definition, for $\lambda\in\mathbb C$, $y\in\mathfrak M(\mathscr A)$
\[
{\rm det}(\hat A_{U_x}(y)-\lambda\cdot \hat{I}_n)={\rm det}(\hat e_{U_x}(\hat A(y)-\hat I_n)+\hat I_n-\lambda\cdot \hat{I}_n)=\left\{
\begin{array}{ccc}
p(\lambda,y)&{\rm if}&y\in U_x\smallskip\\ 
(1-\lambda)^n&{\rm if}&y\not\in U_x.
\end{array}
\right.
\]
Thus the spectrum $\sigma(A_{U_x})$ of $A_{U_x}$ is $(\sigma(A)\cap U_x)\cup\{1\}$. In particular,  $\sigma(A_{U_x})\subset N_s\cup\{1\}$ and so  the unbounded connected component of $\mathbb C\setminus \sigma(A_{U_x})$ contains the open connected set $\mathbb C\setminus (N_s\cup\{1\})$ which in turn contains $0$. Hence, $A_{U_x}\in\Sigma_n(\mathscr A)$.

Next, we cover $\mathfrak M(\mathscr A)$ by sets $U_x$, $x\in \mathfrak M(\mathscr A)$,  choose a finite subcover of this cover and then a refinement of this subcover by clopen mutually disjoint sets, say, $U_1,\dots, U_k$. Let $e_i\in\mathscr A$ be the idempotent such that $\hat e_i$ is the characteristic function of $U_i$. Then $\sum_{i=1}^n e_i=1$ and $e_ie_j=0$ for all $i\ne j$. Also, by our construction all matrices $A_i:=e_{i}(A- I_n)+I_n\in\Sigma_n(\mathscr A)$.
\end{proof}
Using Lemma \ref{lem2.2} let us finish the proof of part (2) of the theorem for $d=0$. 

Let $B_i\in M_n(\mathscr A)$ be a logarithm of $A_i\in\Sigma_n(\mathscr A)$ defined by holomorphic functional calculus (see, e.g., \cite[Thm.\,10.30]{R}). Then $B_i$ belongs to the closure of the subspace of $M_n(\mathscr A)$ generated by matrices $(A_i-\lambda I_n)^{-1}$, $\lambda\not\in\sigma(A_i)$. In particular, all matrices $B_i$, $1\le i\le k$, commute (because matrices $A_i-\lambda I_n$ and $A_j-\gamma I_n$ commute for all $\lambda,\gamma\in\mathbb C$, $i\ne j$). We set
\[
B=\sum_{i=1}^k B_i\in  M_n(\mathscr A).
\]
Then since all $B_i$ commute,
\[
\exp\, B=\prod_{i=1}^k\exp\, B_i=\prod_{i=1}^k A_i=\prod_{i=1}^k (e_{i}(A- I_n)+I_n)=I_n+\sum_{i=1}^n e_{i}(A- I_n)=A.
\]
This shows that $e_n(\mathscr A)=1$ for all $n$ which proves part (2) of Theorem \ref{te1} for $d=0$.\smallskip

Further, part (3) of the theorem follows from Lemma \ref{lem2.1} and \cite[Lm.\,9]{DV} asserting that if $A$ has the Bass stable rank $1$, then $t_n(A)=4$ for all $n\ge 2$.\smallskip

Finally, let us prove part (4). We will prove that $e_n(\mathscr A)\ge 2$ provided that $d:=dim\, \mathfrak M(\mathscr A)\ne 0$. Since clearly $e_n(\mathscr A)\ge e_n(\mathscr A_s)$,
we may assume that the algebra $\mathscr A$ is semisimple and is a (not necessarily closed) subalgebra of $C(\mathfrak M(\mathscr A))$.
We adapt the idea of \cite{XT}. 

Since $d\ne 0$, the compact space $\mathfrak M(\mathscr A)$ is not totally disconnected; hence, it has a connected component $X$ containing two distinct points, say, $x_1$ and $x_2$.
Let $f\in \mathscr A$ be such that $f(x_1)=0$, $f(x_2)=2\pi i$ and $g={\rm exp}\, f$. Let
\[
T=\left(
\begin{array}{cc}
g&1\\
0&1
\end{array}
\right)\in Exp_2(\mathscr A).
\]
We claim that $T$ does not have a logarithm in $M_2(\mathscr A)$.  Indeed, if it is false, then 
there exists 
\[
S=\left(
\begin{array}{cc}
f_1&f_2\\
f_3&f_4
\end{array}
\right)\in M_2(\mathscr A)
\]
such that $S^2=T$, i.e.,
\[
\left(
\begin{array}{cc}
f_1^2+f_2f_3&f_1f_2+f_2f_4\\
f_3f_1+f_4f_3&f_3f_2+f_4^2
\end{array}
\right)=
\left(
\begin{array}{cc}
g&1\\
0&1
\end{array}
\right).
\]
This implies that 
\begin{equation}\label{e2.5}
(f_1+f_4)f_2=1\quad{\rm and}\quad (f_1+f_4)f_3=0,
\end{equation}
i.e., $f_3=0$. Hence, $f_1^2=g$ and $f_4^2=1$.
Since $X$ is connected and $f_1|_X, f_2|_X\in C(X)$, these equations imply that either $f_4|_X=1$ or $f_4|_X=-1$ and either $f_1|_X=\exp\, \bigl(\frac{f}{2}\bigr)|_X$ or $f_1|_X=-\exp\, \bigl(\frac{f}{2}\bigr)|_X$. Let us consider each of these cases.

(a) If $f_1|_X=\exp\, \bigl(\frac{f}{2}\bigr)|_X$ and $f_4|_X=-1$, or
$f_1|_X=-\exp\, \bigl(\frac{f}{2}\bigr)|_X$ and $f_4|_X=1$, then
 $f_1(x_1)+f_4(x_1)=\pm 1+\mp 1=0$ which contradicts the first equation in \eqref{e2.5}.

(b) If $f_1|_X=\exp\, \bigl(\frac{f}{2}\bigr)|_X$ and $f_4|_X=1$, or $f_1|_X=-\exp\, \bigl(\frac{f}{2}\bigr)|_X$ and $f_4|_X=-1$, then $f_1(x_2)+f_4(x_2)=\pm e^{\pi i}+\pm 1=0$ which contradicts the first equation in \eqref{e2.5} as well.

This shows that the above $S\in M_2(\mathscr A)$ does not exist and so $T$ does not have a logarithm in $M_2(\mathscr A)$, i.e., $e_2(\mathscr A)\ge 2$.

Let us prove the result for $n>2$.  Let $M:=1+\max_{\mathfrak M(\mathscr A)} |g|$.
We put
\[
T_n=\left(
\begin{array}{cc}
M\cdot I_{n-2}&0_{(n-2)\times 2}\\
0_{2\times (n-2)}&T
\end{array}
\right)\in {\rm Exp}_n(\mathscr A);
\]
here $0_{k\times l}$ is the $k\times l$ zero matrix.

Assume that $T_n$ has a logarithm in $M_n(\mathscr A)$. Then
there exists 
\[
S_n=\left(
\begin{array}{cc}
L_1&L_2\\
L_3&L_4
\end{array}
\right)\in M_n(\mathscr A)
\]
such that $S_n^2=T_n$; here $L_1\in M_{n-2}(\mathscr A)$, $L_4\in M_2(\mathscr A)$, and $L_2$ and $L_3$ are matrices with entries in $\mathscr A$ of sizes $(n-2)\times 2$ and $2\times (n-2)$, respectively.

The identity $S_nT_n=T_nS_n$ implies that
\[
L_2(T-M\cdot I_2)=0\quad {\rm and}\quad (T-M\cdot I_2)L_3=0.
\]
Since $T-M\cdot I_2$ is invertible by our choice of $M$, we get from here that $L_2=0$ and $L_3=0$. Hence, $L_4^2=T$ which contradiction the first part of the proof. Thus such $S_n\in M_n(\mathscr A)$ does not exist and so $T_n$ does not have a logarithm in $M_n(\mathscr A)$, i.e., $e_n(\mathscr A)\ge 2$ for all $n\ge 3$ as well.

The proof of the theorem is complete.
\end{proof}

\sect{Proof of Theorem \ref{te2.5}} 
\begin{proof}
First, let us proof the theorem for a diagonal matrix $A$ with diagonal entries $a_{11},\dots, a_{nn}\in GL_1(\mathscr A)$ satisfying $a_{11}\cdots a_{nn}=1$. In this case as in the proof of \cite[Lm.\,16]{DV} we have
\begin{equation}\label{equ3.1}
A=C_A^{-1} R_n^{-1} C_A R_n,
\end{equation}
where $C_A=(c_{ij})$ is a diagonal matrix with $c_{11}=1, c_{22}=a_{11}, c_{33}=a_{11}a_{22},\dots, c_{nn}=a_{11}\cdots a_{n-1\, n-1}$ and 
$R_n$ is a (cyclic) permutation operator on $\mathscr A^n$ such that $R_n(e_i)=e_{i+1}$,
$1\le i\le n-1$, and $R_n(e_n)=e_1$; here $e_i\in\mathscr A^n$ is the vector having the $i$th coordinate $1$ and all others $0$. By the definition, the spectrum $\sigma_n(R_n)$ of $R_n$ is $S_n$ (the set of roots of order $n$ of $1$). Hence, $R_n$ has a logarithm $R_n'\in M_n(\mathscr A)$ (cf. \cite[Thm.\,10.30]{R}). We set $B_1=-C_A^{-1} R_n' C_A$ and $B_2=R_n'$. Then $A=C_A^{-1} R_n^{-1} C_A R_n=\exp\, B_1\cdot \exp\, B_2$  and $\sigma_n(\exp\, B_1)=\sigma_n (\exp\, B_2)=S_n$, as required. 

Now, let us prove the statement in the general case. Without loss of generality assume that $A$ is upper triangular. We denote by $D(A)$ the diagonal matrix whose diagonal is the same as that of $A$ and let $D(A)=C_{D(A)}^{-1}R_n^{-1}C_{D(A)}R_n$ be factorization \eqref{equ3.1} for $D(A)$.
 Let $D_n(t)\in GL_n(\mathscr A)$, $t\in\mathbb C\setminus\{0\}$, be the diagonal matrix with diagonal entries $1,t,\dots, t^{n-1}$. Then the upper triangular matrix $D_n(t)^{-1}AD_n(t)$ is equal to $C_{D(A)}^{-1}R_n^{-1}C_{D(A)}R_nA(t)$, where $A(t)$ is an upper triangular unipotent matrix such that $\lim_{t\rightarrow 0}A(t)=I_n$ (convergence in $M_n(\mathscr A)$). In particular, according to  \cite[Thm.\,10.20]{R}
 there is some $t_\varepsilon>0$ such that $\sigma_n(R_nA(t_{\varepsilon}))\subset N_\varepsilon$ (the closed $\epsilon$-neighbourhood of $S_n$). Without loss of generality we may assume that $\varepsilon$ is so small that $\mathbb C\setminus N_\varepsilon$ is an open connected set containing $0$. Then $R_nA(t_{\varepsilon})$ has a logarithm $B_{t_\varepsilon}\in M_n(\mathscr A)$ (cf. \cite[Thm.\,10.30]{R}).
 We set $B_1:=-D_n(t_\varepsilon)C_{D(A)}^{-1}R_n'C_{D(A)}D_n(t_\varepsilon)^{-1}$, where $R_n'$ is a logarithm of $R_n$, and $B_2:=D_n(t_\varepsilon)B_{t_\varepsilon}D_n(t_\varepsilon)^{-1}$.
  Clearly, 
 \[
A= \exp\, B_1\cdot \exp\, B_2\quad{\rm and}\quad \sigma_n(\exp\, B_1)=S_n,\quad \sigma_n(\exp B_2)\subset N_\varepsilon.
\]
The proof of the theorem is complete.
\end{proof}

\sect{Proof of Theorem \ref{te2}}
First, we prove the following auxiliary result.

Let $\mathscr A$ be an algebra satisfying condition (*) of Theorem \ref{te2}. We denote by $M_{k,l}(\mathscr A)$ the set of matrices of size $k\times l$ with entires in $\mathscr A$.
Let 
\[
U_n(\mathscr A):=\{A\in M_{n,1}(\mathscr A)\, :\, \exists\, B\in M_{1,n}(\mathscr A)\ {\rm such\ that}\ BA=1\}
\] 
be the subset of left-invertible matrices in $M_{n,1}(\mathscr A)$.
\begin{Lm}\label{lem2.6}
Suppose $A=(a_1,\dots ,a_n)^{\mathrm T}\in U_n(\mathscr A)$, $n\ge 2$, where one of $a_2,\dots, a_n$ is nonzero. Then there exists a unipotent upper triangular matrix $C\in GL_n(\mathscr A)$ such that
\[
CA=(a_1',\dots ,a_{n-1}',a_n)^{\mathrm T},
\]
where $a_1'\in {\rm Exp}_1(\mathscr A)$.
\end{Lm}
\begin{proof}
We prove this result by induction on $n$. For $n=2$, since $a_2\ne 0$, and $\hat a_1|_{Z_{a_2}}\in GL_1(C(Z_{a_2}))$ has a logarithm in $C(Z_{a_2})$ due to condition (*), there is $b\in \mathscr A$ such that $a_1+ba_2=a_1'\in {\rm Exp}_1(\mathscr A)$ (see, e.g., \cite[Cor.\,4.13]{MR1}). 
Thus  the lemma holds with
\[
C=\left(
\begin{array}{cc}
1&b\\
0&1
\end{array}
\right).
\]

Assume that the lemma is valid for all natural numbers $\le n-1$, $n\ge 3$; let us prove it for $n$.
Without loss of generality we may assume that $a_n\ne 0$ as for otherwise the result follows from the induction hypothesis. Since the Bass stable rank of $\mathscr A$ is $1$  and $A\in U_n(\mathscr A)$,
there exist $b_1,\dots, b_{n-1}\in \mathscr A$ such that $A_1:=(a_1+b_1a_n,\dots, a_{n-1}+b_{n-1}a_n)^{\mathrm T}\in U_{n-1}(\mathscr A)$. Clearly, for
\[
C_1=\left(
\begin{array}{cc}
I_{n-1}&B\\
0_{n-1}&1
\end{array}
\right),\qquad B=(b_1,\dots , b_{n-1})^{\mathrm T},\quad 0_{n-1}:=0\in M_{1,n}(\mathscr A),
\]
we have 
\[
C_1A=(a_1+b_1a_n,\dots , a_{n-1}+b_{n-1}a_n, a_n)^{\mathrm T}.
\]
We may assume that one of $a_i+b_ia_n$, $2\le i\le n-1$, is nonzero (for otherwise $C_1A=(a_1+b_1a_n,0\dots , 0, a_n)^{\mathrm T}$ and the required result follows as in the case $n=2$).
Then by the induction hypothesis, there is a unipotent upper triangular matrix $C_2\in GL_{n-1}(\mathscr A)$ such that $C_2A_1=(a_1',\dots ,a_{n-2}',a_{n-1}+b_{n-1}a_n)^{\mathrm T},$ where $a_1'\in {\rm Exp}_1(\mathscr A)$. Thus  the statement of the lemma holds with
\[
C=\widetilde C_2 C_1,\quad{\rm where}\quad 
\widetilde C_2=\left(
\begin{array}{cc}
C_2&0_{n-1}^{\rm T}\\
0_{n-1}&1
\end{array}
\right).
\]
This proves the induction step and hence the lemma.
\end{proof}
\begin{proof}[Proof of Theorem \ref{te2}]
We prove that for each $A\in {\rm Exp}_n(\mathscr A)$ there exist $B_1,B_2\in M_n(\mathscr A)$ such that $A=\exp\, B_1\cdot \exp\, B_2$. Note that since the Bass stable rank of $\mathscr A$ is $1$, a matrix $A\in Exp_n(\mathscr A)$ if and only if ${\rm det}\, A\in {\rm Exp}_1(\mathscr A)$ (see \cite[Lm.\,9]{DV}). Hence,
${\rm Exp}_n(\mathscr A)$ is the direct product of subgroups ${\rm Exp}_1(\mathscr A)\cdot I_n$ and $SL_n(\mathscr A)$.

We prove the result by induction on $n$. Since the result is obvious for $n=1$, we assume that it is valid for all
natural numbers $\le n-1$, $n\ge 1$, and now prove it for $n$. 

So let $A=(a_{ij})\in {\rm Exp}_n(\mathscr A)$. We consider two cases:\smallskip

\noindent (1) $A$ is an upper triangular matrix. Then the required result follows straightforwardly from Theorem \ref{te2.5}.\smallskip

\noindent (2) There exists $a_{ij}\ne 0$ for some $j< i$. 

Let $S_{ij}$ be the constant elementary matrix such that
 $AS_{ij}$ is obtained from $A$ by interchanging the first and the $j$th  columns. Then the matrix $(a_{ij}'):=S_{ij}^{-1}A S_{ij}$ is such that $a_{i1}'=a_{ij}\ne 0$.  Clearly,  it suffices to prove the result for such a matrix. Thus without loss of generality we may assume that $A$ itself is such that $c_{i1}\ne 0$ for some $i\ge 2$. Then according to Lemma \ref{lem2.6}, there is a unipotent upper triangular matrix 
 $B\in GL_n(\mathscr A)$ such that  the top-left corner  entry of $A_1=BA$ is in ${\rm Exp}_1(\mathscr A)$. Then the top-left corner entry of the matrix $BAB^{-1}=A_1B^{-1}=:A_2$ is the same as that of $A_1$.
 Multiplying $A_2=(\tilde a_{ij})$ from the right by an upper triangular matrix $D$ whose first row is $(1,-\frac{\tilde a_{12}}{\tilde a_{11}},\dots, -\frac{\tilde a_{1n}}{\tilde a_{11}})$ and other rows are as in $I_n$, we obtain a matrix $A_3=A_2D$ whose first row is $(\tilde a_{11},0,\dots, 0)$. Hence, it suffices to prove the result for $A_2:=A_3D^{-1}$. By definition,
 \[
 A_2=\left(
 \begin{array}{cc}
 \tilde a_{11}&0_{n-1}\\
 H&G
 \end{array}
 \right)\cdot 
 \left(\begin{array}{cc}
 1&K\\
 0_{n-1}^{\rm T}&I_{n-1}
 \end{array}
 \right)
 \]
 for some $H\in M_{n-1,1}(\mathscr A)$, $G\in GL_{n-1}(\mathscr A)$,
 $K:=(\frac{\tilde a_{12}}{\tilde a_{11}},\dots, \frac{\tilde a_{1n}}{\tilde a_{11}})\in M_{1,n-1}(\mathscr A)$. Since ${\rm det}\, G=\frac{{\rm det}\, A}{\tilde a_{11}}\in {\rm Exp}_{1}(\mathscr A)$, the matrix $G\in {\rm Exp}_{n-1}(\mathscr A)$ and, hence,
by the induction hypothesis, $G$ can be written as $G=G_1\cdot G_2$ for some $G_i\in GL_{n-1}(\mathscr A)$ having logarithm. Thus for some $\lambda>0$
 \begin{equation}\label{equ4.1}
 \begin{array}{l}
 \displaystyle
 A_2=\left(
 \begin{array}{cc}
 \tilde a_{11}&0_{n-1}\\
 H&\lambda\cdot G_1
 \end{array}
 \right)\cdot 
 \left(\begin{array}{cc}
 1&0_{n-1}\\
 0_{n-1}^{\rm T}&\frac{1}{\lambda}\cdot G_2
 \end{array}
 \right)\cdot 
 \left(\begin{array}{cc}
 1&K\\
 0_{n-1}^{\rm T}&I_{n-1}
 \end{array}
 \right)=\\
 \\
 \displaystyle \left(
 \begin{array}{cc}
 \tilde a_{11}&0_{n-1}\\
 H&\lambda\cdot G_1
 \end{array}
 \right)\cdot \left(\begin{array}{cc}
 1&K\\
 0_{n-1}^{\rm T}&\frac{1}{\lambda}\cdot G_2
 \end{array}
 \right).
 \end{array}
 \end{equation}
Let us choose $\lambda$ so large that the spectra $\sigma_{n-1}(\lambda\cdot G_1)$ and
$\sigma_1(\tilde a_{11})$ of $\lambda\cdot G_1\in GL_{n-1}(\mathscr A)$ and $\tilde a_{11}\in GL_1(\mathscr A)$ do not intersect and the spectrum $\sigma_{n-1}(\frac{1}{\lambda}\cdot G_2)
$ of $\frac{1}{\lambda}\cdot G_2\in GL_{n-1}(\mathscr A)$ does not contain $1$. Then the matrices $\lambda\cdot G_1-\tilde a_{11}\cdot I_{n-1}$ and $I_{n-1}-\frac{1}{\lambda}\cdot G_2$ belong to $GL_{n-1}(\mathscr A)$ and for $X:=(\lambda\cdot G_1-\tilde a_{11}\cdot I_{n-1})^{-1}H\in M_{n-1,1}(\mathscr A)$ and $Y:= K(I_{n-1}-\frac{1}{\lambda}\cdot G_2)^{-1}\in  M_{1,n-1}(\mathscr A)$ we have
\begin{equation}\label{equ4.2}
\begin{array}{l}
\displaystyle
\left(
 \begin{array}{cc}
 1&0_{n-1}\\
X & I_{n-1}
 \end{array}
 \right)\cdot
\left(
 \begin{array}{cc}
 \tilde a_{11}&0_{n-1}\\
 H&\lambda\cdot G_1
 \end{array}
 \right)\cdot \left(
 \begin{array}{cc}
 1&0_{n-1}\\
X & I_{n-1}
 \end{array}
 \right)^{-1}=\left(
 \begin{array}{cc}
 \tilde a_{11}&0_{n-1}\\
0_{n-1}^{\mathrm T} & \lambda\cdot G_1
 \end{array}
 \right),
 \end{array}
 \end{equation}
\begin{equation}\label{equ4.3}
\begin{array}{l}
\displaystyle
\left(
 \begin{array}{cc}
 1&Y\\
0_{n-1}^{\mathrm T} & I_{n-1}
 \end{array}
 \right)\cdot
\left(
 \begin{array}{cc}
 1&K\\
 0_{n-1}^{\mathrm T}&\frac{1}{\lambda}\cdot G_2
 \end{array}
 \right)\cdot \left(
 \begin{array}{cc}
 1&Y\\
0_{n-1}^{\mathrm T}&I_{n-1}
 \end{array}
 \right)^{-1}=\left(
 \begin{array}{cc}
 1&0_{n-1}\\
0_{n-1}^{\mathrm T} &\frac{1}{\lambda}\cdot G_2
\end{array}
 \right).
 \end{array}
 \end{equation}
Clearly, the matrices on the right-hand sides of \eqref{equ4.2} and \eqref{equ4.3} have logarithm. Hence, matrix $A_2$ in \eqref{equ4.1} is a product of two exponents of matrices from $M_n(\mathscr A)$. This completes the proof of the induction step and therefore of the theorem.
\end{proof}

\end{document}